\documentclass[11pt]{amsart}
\date{\today}
\usepackage{graphicx, color}
\usepackage{verbatim}
\usepackage{amsmath}
\usepackage{amssymb, mathrsfs}
\usepackage{amsbsy}
\usepackage{amscd}
\usepackage{amsthm}
\usepackage{bm}
\usepackage{epic,eepic}

\newcommand{\Q}{\mathbb{Q}}
\newcommand{\N}{\mathbb{N}}
\newcommand{\R}{\mathbb{R}}
\newcommand{\Z}{\mathbb{Z}}
\newcommand{\E}{\mathbb{E}}
\newcommand{\set}[1]{\left\{#1\right\}}

\newcommand{\geo}{\mathop{\mathtt{GEO}}\nolimits}
\renewcommand{\P}{\mathbb{P}}

\newcommand{\ket}{\right\rangle}
\newcommand{\bra}{\left\langle}

\newcommand{\dd}{{\rm{d}}}

\newtheorem{theorem}{Theorem}
\newtheorem*{theorem*}{Theorem}
\newtheorem{lemma}{Lemma}

\newtheorem*{proposition*}{Proposition}

\theoremstyle{definition}
\newtheorem{remark}{Remark}

\newtheorem{definition}{Definition}

\numberwithin{equation}{section}

\begin{document}

\title[Jump size in the directed first passage percolation]
{Note on the maximal jump size in a continuum model of directed first passage percolation}
\author{Ryoki Fukushima}
\address{Research Institute for Mathematical Sciences}
\email{ryoki@kurims.kyoto-u.ac.jp}


\begin{abstract}
In this note, we consider the directed first passage percolation introduced in~[F.~Comets, R.~Fukushima, S.~Nakajima and N.~Yoshida: Journal of Statistical Physics, 161-(3), 577--597 (2015)]. It is proved that the shortest path from the origin to the $n$-th hyperplane makes a jump larger than a positive power of $\log n$. Some numerical results are also provided, which indicates that the maximal jump size is much larger in a certain parameter region. 
\end{abstract}

\maketitle

\section{Introduction and result}
We study a geometric property of the the shortest path (geodesic) in a continuum model of directed first passage percolation introduced in~\cite{CFNY15}. Specifically, we are interested in the maximal jump size in the geodesic in this model. The result in this note implies that it grows unbounded as the system size goes to infinity. In the end of this article, some numerical results are also presented. 

We start by recalling the first passage percolation model in~\cite{CFNY15}. Let $(\omega, \P)$ be the Poisson point process on $\N\times \R^d$ whose intensity is the product of the counting measure and the Lebesgue measure. 
It is natural to realize it as a sum of the Dirac measures but with some abuse of notation, we will often identify $\omega$ with its support.
For a given sample $\omega$, we call $\gamma=(\gamma(k))_{k=1}^n\in (\R^d)^n$ an open path (with length $n$) if $\{(k,\gamma(k))\}_{k=1}^n$ is a subset of $\omega$ and we define its passage times by 
\begin{equation*}
T_\gamma(\omega)=\sum_{k=1}^n|\gamma(k-1)-\gamma(k)|^\alpha.
\end{equation*}
The passage time from $0$ to the $n$-th section $\{n\}\times \R^d$ is defined by
\begin{equation} 
\label{passtime}
T_n(\omega)
=\inf\set{T_\gamma(\omega)\colon \gamma(0)=0 \textrm{ and }
\gamma\textrm{ is open with length }n} 
\end{equation} 
and $T_{v,w}(\omega)$ for space-time points $v,w$ is defined similarly. This model can be regarded as a directed version of the Euclidean first passage percolation introduced and studied in~\cite{HN97,HN99,HN01}. 

As in other first passage percolation models (see, for example, the monograph~\cite{ADH17}), the following two problems are of interest: (i) asymptotic behaviors of $T_n$ and (ii) properties of the minimizing path. For the first problem (i), a direct application of the subadditive ergodic theorem shows that the so-called \emph{time constant}
\begin{equation} 
\mu=\lim_{n\to\infty}{1 \over n}T_n(\omega)
\label{mu}
\end{equation} 
exists $\P$-almost surely. This limit is deterministic and shown to be positive in~\cite{CFNY15}. The rate of convergence is a more interesting but challenging question. The works~\cite{CFNY15,Nak18} contain some results in this direction but they do not seem to be optimal. 
The problem (ii) concerns the geodesic, denoted by $\geo_{n;\omega}$, which is defined by the open path $\gamma$ with length $n$ for which $T_\gamma(\omega)=T_n(\omega)$. A point-to-point geodesic $\geo_{v,w;\omega}$ is defined similarly. Note that these geodesics are almost surely uniquely determined in our continuum setting if $\alpha\neq 1$. In a recent work~\cite[Corollary~2.7]{Nak18}, it is proved that for any $\alpha > 1$, the maximal jump size in $\geo_{n;\omega}$ is of order  $n^{o(1)}$ as $n\to \infty$. On the other hand, a numerical experiment presented in Section~\ref{sec:numerics} indicates that the maximal jump size is much larger when $\alpha$ is small. It is an interesting problem to determine the asymptotics of the maximal jump size. The following result is a modest step toward this direction, which in particular shows that the maximal jump size is unbounded even for $\alpha>1$.  
\begin{theorem}
\label{thm:main}
For any $d\in\N$ and $\alpha\neq 1$, there exists $\epsilon>0$ such that
$\P$-almost surely, the maximal jump of $\geo_{n;\omega}$ is larger than 
$(\log n)^\epsilon$ for sufficiently large $n$. 
\end{theorem}
\begin{remark}
This result is analogous to the one in~\cite{BK93} for the classical lattice first passage percolation, which states (as a special case) that the geodesic goes through an edge with an arbitrarily large passage time as soon as the distribution of the edge weight is unbounded. In this lattice case, a more detailed result has recently been obtained in~\cite{Nak19+}. 
\end{remark}
\begin{remark}
The case $\alpha=1$ is somewhat singular and not covered in this note. For instance, the geodesic is not necessarily unique in $d=1$. 
\end{remark}

\noindent\textbf{Notational convention}\\
For a time-space point $w\in \N\times \R^d$, we write its time coordinate as $w_{\rm t}$ and space coordinate as $w_{\rm s}$. 
\section{Proof of Theorem~\ref{thm:main}}
The basic strategy is similar to the work~\cite{BK93},
with the help of an idea from~\cite{DCKNPS}.  
Fix four positive constants $c_1>0$, $0<\epsilon<\beta$ and $M>0$ whose values will be determined later. 
For $(k,x)\in\N\times4\Z^d$, let us define a \emph{face} by
\begin{equation*}
 F_n(k,x)=\{k(\log n)^\beta \}\times 
\left(x+[-2,2)^d\right). 
\end{equation*}
\begin{definition}
\label{def:black}
A \emph{face} $F_n(k,x)$ is \emph{black} if the following holds for all $v\in F_n(k,x)$:
\begin{enumerate}
 \item[\textbf{(i)}] when $\alpha<1$, 
\begin{equation}
\inf_{w\in\{(k+1)(\log n)^\beta\}\times\Z^d} T_{v,w}(\omega) \ge c_1 (\log n)^\beta,
\label{eq:black<1}
\end{equation}
 \item[\textbf{(ii)}] when $\alpha>1$, for all $w\in \{(k+1)(\log n)^\beta\}\times\Z^d$ satisfying $|x-w_{\rm s}|\le 2M(\log n)^\beta$, 
\begin{equation}
T_{v,w}(\omega) \ge 
\left(\left(\frac{|v_{\rm s}-w_{\rm s}|}{(\log n)^{\beta}}\right)^\alpha+c_1\right)
(\log n)^\beta.
\label{eq:black>1}
\end{equation}
\end{enumerate}
\end{definition}
\begin{lemma}
\label{black}
Let $\alpha\neq 1$ and $\beta>0$. Then for all sufficiently small $c_1>0$, 
\begin{equation*}
\P(F_n(k,x)\textrm{ is black})\ge 1-\exp\{-(\log n)^{c_1}\}.
\end{equation*} 
\end{lemma}
\begin{proof}
We may assume $(k,x)=(0,0)$ without loss of generality and write $F_n$ for $F_n(0,0)$. 

Let us first deal with the case $\alpha<1$. Note that for $k=0$ and fixed $v$, the left hand side of~\eqref{eq:black<1} has the same law as $T_{(\log n)^\beta}(\omega)$. 
Recall the tail estimate proved in~\cite[Proposition~3.1]{CFNY15}: for any $\lambda<1/2$, there exist $C_1, C_2>0$ such that for all $N\in\N$, 
\begin{equation}
\P\left(T_N(\omega)-N\mu<-N^{1-\lambda}\right)
\le C_1\exp\left\{-C_2N^\lambda \right\}.
\label{concentration3}
\end{equation}
Choosing $0<c_1<(\mu\wedge\beta)/4$ and $N=[(\log n)^\beta]$, 
one finds that 
\begin{equation}
\P\left(T_{(\log n)^\beta}(\omega) <2c_1 (\log n)^\beta\right)
\le \exp\{-2(\log n)^{\beta}\}. 
\label{eq:v=0}
\end{equation}
While this is for a fixed $v$, the extension to all $v\in F_n$ is straightforward: 
if $\inf_w T_{v,w}(\omega) < \epsilon(\log n)^\beta$ for some $v\in F_n$, then by considering the path starting at $0$ and following the geodesic for $\inf_w T_{v,w}(\omega)$ after the first step, we find that 
\begin{equation*}
T_{(\log n)^\beta}(\omega) \le c_1(\log n)^\beta+2\sqrt{d}
\end{equation*}
by using the triangle inequality $|x+y|^\alpha\le|x|^\alpha+|y|^\alpha$ 
for $\alpha< 1$. Thus this case is covered by~\eqref{eq:v=0}. 

Next we consider the case $\alpha>1$. In this case, the shortest path 
from $v=(v_{\rm t},v_{\rm s})$ to $w=(w_{\rm t},w_{\rm s})$ is, if the open 
path constraint is dropped, the straight
line with the direction $D=(w_{\rm s}-v_{\rm s})(\log n)^{-\beta}$. Suppose that
\begin{equation}
T_{v,w}(\omega) <\left(|D|^\alpha
+c_1\right)(\log n)^\beta.
\label{white}
\end{equation}
When $|D|<\delta$ with $\delta^\alpha+2\epsilon<\mu/2$, then the probability
of this event is controlled in the same way as before. 
For the case $|D|\ge \delta$, we apply an affine tilting to the configuration $\omega$ to define 
$\tilde{\omega}$ so that $v\mapsto(0,0)$ and $w\mapsto((\log n)^\beta,0)$, 
that is, we translate each section as
$\tilde\omega|_{\{j\}\times \R^d}=\omega|_{\{j\}\times \R^d}-v_{\rm s}-jD$. 
Then, since $\{(j,\geo_{v,w; \omega}(j)-j D)\}_{j=1}^{(\log n)^\beta}\subset \tilde\omega$, 
we have
\begin{equation*}
 T_{(\log n)^\beta}(\tilde\omega)
\le \sum_{j=1}^{(\log n)^\beta}
 |\geo_{v,w; \omega}(j-1)-\geo_{v,w; \omega}(j)-D|^\alpha\\
\end{equation*} 
We show that this is smaller than
$\mu(\log n)^\beta/2$ under the assumption~\eqref{white}. 
Let $\Delta_j=\geo_{v,w; \omega}(j-1)-\geo_{v,w; \omega}(j)-D$. 
By using the Taylor expansion, we get 
\begin{equation}
\begin{split}
&T_{v,w}(\omega) \\
&\quad=\sum_{j=1}^{(\log n)^\beta}
 |\geo_{v,w; \omega}(j-1)-\geo_{v,w; \omega}(j)|^\alpha\\
&\quad\ge \sum_{j=1}^{(\log n)^\beta}
 \bigl[|D|^\alpha+\bra\nabla |D|^\alpha,\Delta_j\ket\\
&\qquad +c\left(|D|^{\alpha-2}|\Delta_j|^21_{\{|\Delta_j|\le |D|/2\}}
+ |\Delta_j|^\alpha 1_{\{|\Delta_j|> |D|/2\}}\right)\bigr].
\end{split}
\label{eq:Taylor}
\end{equation}
Note that the second term sums up to zero since $\sum_{k=1}^{(\log n)^\beta}\Delta_j=0$. Recalling $\delta\le|D|\le M$, we have 
\begin{equation*}
 \delta^{\alpha-2}\wedge M^{\alpha-2}\le |D|^{\alpha-2}\le \delta^{\alpha-2}\vee M^{\alpha-2},
\end{equation*}
and hence~\eqref{white} and~\eqref{eq:Taylor} imply 
\begin{equation*}
\begin{split}
& \sum_{j=1}^{(\log n)^\beta}
 |\Delta_j|^\alpha\\
 &\quad\le \sum_{j=1}^{(\log n)^\beta}
 \left(\delta^\alpha+
 \delta^{\alpha-2}|\Delta_j|^21_{ \{\delta\le|\Delta_j|\le |D|/2\}}
 + |\Delta_j|^\alpha 1_{\{|\Delta_j|> |D|/2\}}\right)\\
 &\quad\le c\left(\delta^\alpha+c_1\left(\frac{M}{\delta}\right)^{|\alpha-2|}+c_1\right)
(\log n)^\beta.
\end{split}
\end{equation*}
The coefficient in the last line can be made smaller than $\mu/2$ by letting 
$c_1$ small. Therefore, for $|D|\ge\delta$, we arrive at 
\begin{equation*}
\P\left(T_{v,w}(\omega) <\left(|D|^\alpha+c_1\right)(\log n)^\beta\right)
\le \P\left( T_{(\log n)^\beta}(\tilde\omega)\le \mu(\log n)^\beta/2\right).
\end{equation*}
Since $\tilde\omega$ has the same law as $\omega$, this right hand side
is again bounded by $\exp\{-(\log n)^c\}$. This bound can be extended to all lattice point $w$ satisfying $|w|\le 2M(\log n)^\beta$ by the union bound. In order to pass from lattice points to continuum, note that under~\eqref{white}, the first and last jumps of the geodesic must be smaller than 
$(M^\alpha+2c_1)^{1/\alpha}(\log n)^{\beta/\alpha}$. Then the costs to
change the first and last steps are bounded by 
$c((\log n)^{\beta/\alpha})^{\alpha-1}=o((\log n)^\beta)$ by the Taylor expansion.
\end{proof}

Let us consider the following conditions:
\begin{gather}
\max_{1\le j\le n}|\geo_{n;\omega}(j-1)-\geo_{n;\omega}(j)|
 \le (\log n)^\epsilon,\label{small-jump}\\
\#\{F_n(k,x)\colon \textrm{black and  crossed by }\geo_{n;\omega}\}
 \ge (1-\epsilon)n(\log n)^{-\beta},\label{many-black}\\
 \#\{k\colon |[\geo_{n;\omega}(k(\log n)^\beta)]-
 [\geo_{n;\omega}((k+1)(\log n)^\beta)]|\ge M(\log n)^\beta\}\label{linear}\\
\quad \le (1_{\{\alpha< 1\}}+\epsilon^2 1_{\{\alpha>1\}})n(\log n)^{-\beta}.
\notag
\end{gather}
The last condition is non-trivial only for $\alpha>1$. 
We first check that under~\eqref{small-jump}, we 
have~\eqref{many-black} and \eqref{linear} with high probability. 
\begin{lemma}
\label{3}
When $M$ is sufficiently large depending on $\epsilon>0$, 
\begin{equation*}
\begin{split}
& \P(\textrm{\eqref{small-jump} holds and either 
 \eqref{many-black} or \eqref{linear} fails})\\
&\quad \le \exp\{-n^{1\wedge\frac{d}{\alpha}}(\log n)^{-\beta}\}.
\end{split}
\end{equation*}
\end{lemma}
\begin{proof}
Let us begin with 
\begin{equation*}
\begin{split}
&\P(\textrm{\eqref{small-jump} holds and \eqref{many-black} fails})\\
&\quad \le \sum_{v_1,\ldots,v_{n(\log n)^{-\beta}}}
\P\left(\#\{k\colon F_n(k,v_k)\textrm{ is black}\}
<(1-\epsilon)n(\log n)^{-\beta}\right),
\end{split}
\end{equation*}
where the sum runs over sequences in $4\Z^d$ satisfying $v_1=0$ and 
$|v_k-v_{k+1}|\le(\log n)^{(\beta+\epsilon)}$. 
There are at most $((\log n)^{d(\beta+\epsilon)})^{n(\log n)^{-\beta}}$ such
sequences and for each of them, the number of possibilities to place less than
$(1-\epsilon)n(\log n)^{-\beta}$ black faces is bounded by
\begin{equation*}
(1-\epsilon)n(\log n)^{-\beta} \binom{n(\log n)^{-\beta}}
{(1-\epsilon)n(\log n)^{-\beta}}
\le \exp\set{n(\log n)^{-\beta}},
\end{equation*}
which is much smaller than the above 
$((\log n)^{d(\beta+\epsilon)})^{n(\log n)^{-\beta}}$. 
As the rest of faces must be white, by using Lemma~\ref{black} we get
\begin{equation*}
\begin{split}
&\P(\textrm{\eqref{small-jump} holds and \eqref{many-black} fails})\\
&\quad \le \left((\log n)^{d\beta}\right)^{2n(\log n)^{-\beta}}
\P(F_n\textrm{ is white})^{\epsilon n(\log n)^{-\beta}}\\
&\quad\le \exp\set{-n(\log n)^{-\beta}}.
\end{split}
\end{equation*}


Next, when $\alpha>1$ and~\eqref{linear} fails to hold,
then $T_n(\omega)>M^\alpha \epsilon^2 n$ since the passage time 
cannot be shorter than the one for the straight line. Thus if we choose $M$ sufficiently large, then 
$\P(\textrm{\eqref{linear} fails})\le \exp\{-cn^{1\wedge \frac{d}{\alpha}}\}$ 
by Lemma~3.2 in~\cite{CFNY15}. 
\end{proof}
Let us define
\begin{equation*}
\begin{split}
\mathcal{E}=\bigl\{\omega\colon &\textrm{\eqref{small-jump}--\eqref{linear} 
 hold and for every $0\le k\le n(\log n)^{-\beta}$,}\\ 
 &\quad \omega\left(\{k(\log n)^\beta\}\times 
B(\geo_{n;\omega}(k(\log n)^\beta), e^{-n})\right)=1\bigr\}
\end{split}
\end{equation*}
and for $\omega\in \mathcal{E}$, we define a re-sampling operation 
as follows (see also Figure~\ref{resampling}):
\begin{enumerate}
\item[\textbf{(i)}] fix $\theta<1$ and choose an ordered subset $(k_1,\ldots,k_{n^\theta})$ of
\begin{equation*}
\begin{split}
 \{k\colon &F_n(k,x) \textrm{ is black and crossed by }\geo_{n;\omega}
 \textrm{ and when }\alpha>1, \\
 &|[\geo_{n;\omega}(k(\log n)^\beta)]-[\geo_{n;\omega}((k+1)(\log n)^\beta)]|
 < M(\log n)^\beta\}
\end{split}
\end{equation*}  
uniformly at random (\eqref{many-black} and \eqref{linear} ensure that there are at least $(1-2\epsilon)n(\log n)^{-\beta}$ such $k$'s), 
\item[\textbf{(ii)}] for each $1\le j\le n^\theta$, 
replace $\omega$ by an independent copy in the tubes with the width $(\log n)^{2\beta}$ around 
\begin{equation}
\bigcup_{l=0}^{(\log n)^\beta}
\set{(k_j(\log n)^\beta+l,[\geo_{n;\omega}(k_j(\log n)^\beta)])}
\label{tube<1}
\end{equation}
when $\alpha<1$, and around the straight lines connecting
\begin{equation}
\begin{split}
& \left(k_j(\log n)^\beta,[\geo_{n;\omega}(k_j(\log n)^\beta)]
\right)\textrm{ and }\\
&\quad \left((k_j+1)(\log n)^\beta,[\geo_{n;\omega}((k_j+1)(\log n)^\beta)]
\right)
\end{split}
\label{tube>1}
\end{equation}
when $\alpha>1$, except for the $e^{-n}$ neighborhoods of 
$\geo_{n;\omega}(k_j(\log n)^\beta)$ and 
$\geo_{n;\omega}((k_j+1)(\log n)^\beta)$ in both cases. 
\end{enumerate}
\begin{figure}[h]
\includegraphics[width=300pt]{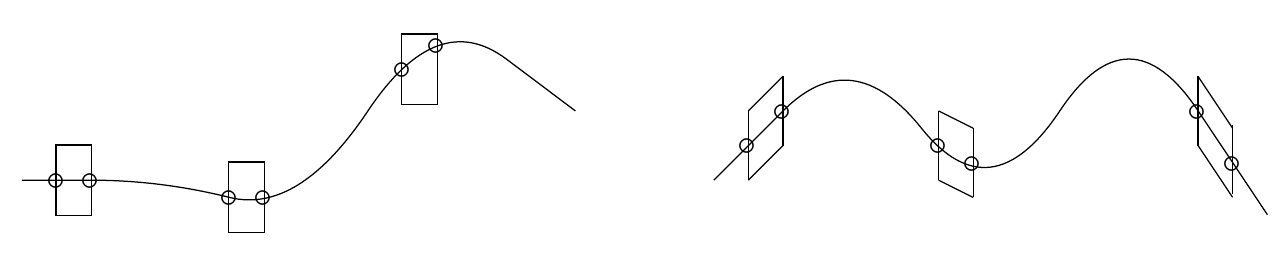}
\caption{Schematic figures of where re-sampling is done. 
Left figure is for $\alpha<1$ and 
right figure is for $\alpha>1$. Curves are $\geo_{n;\omega}$ and the rectangles
(left) and parallelograms (right), with $\circ$'s omitted, are $\sigma_\omega$.}
\label{resampling} 
\end{figure}
We write $\sigma_\omega$ ($\omega\in\mathcal{E}$) for the collection of
tubes where the re-sampling is done, $\tilde\omega$ for the independent
copy used in~(ii), and 
$[\omega_1, \omega_2]_S=\omega_1 1_{S^c}+\omega_2 1_S$
for $S\subset \N\times\R^d$. 
By using these terms, the result of the above re-sampling is written as
$[\omega,\tilde\omega]_{\sigma_\omega}$.
For a given~$\omega$, let $\Sigma_\omega$ denotes the set of possible 
choices of $\sigma_\omega$ and $\Q_\omega$ the (conditional) 
law of $\sigma_\omega$. By choosing $\epsilon$ much smaller than $\theta$, 
\eqref{many-black} and \eqref{linear} imply 
\begin{equation}
\label{cardinality}
 \#\Sigma_\omega\ge \binom{(1-2\epsilon)n(\log n)^{-\beta}}{n^\theta}
\ge \exp\set{(1-2\theta)n^\theta\log n}
\end{equation}
for sufficiently large $n$ uniformly in $\omega\in\mathcal{E}$. 
Finally, the joint law of 
$(\omega,\sigma_\omega,\tilde\omega)$ is denoted by $\bar\P$.
\begin{definition}
\label{tunnel}
For a tube of the form~\eqref{tube<1} or~\eqref{tube>1}, let 
\begin{align*}
D_{k_j}=
\begin{cases}
0, &\alpha<1,\\
(\log n)^{-\beta}
(\geo_{n;\omega}((k_j+1)(\log n)^\beta)-\geo_{n;\omega}(k_j(\log n)^\beta)), 
&\alpha> 1.
\end{cases}
\end{align*}
The tube is \emph{tunneling} for $\tilde\omega$ if the following conditions are satisfied: for every $1 \le l<(\log n)^\beta$,
\begin{equation*}
\begin{split}
& \{k_j(\log n)^\beta+l\}\times \\
&\quad \left(B\left(\geo_{n;\omega}(k_j(\log n)^\beta)
 + l D_{k_j}+\frac{3}{2}\mathbf{e}_1
 (\log n)^\epsilon1_{\{l=(\log n)^\beta/2\}},c_1^{2/(\alpha\wedge 1)}
 \right) \right)
\end{split}
\end{equation*} 
contains exactly one point of $\tilde\omega$, and otherwise there is no point in
\begin{equation*}
\bigcup_{l=0}^{(\log n)^\beta} 
\{k_j(\log n)^\beta+l\}\times ( B([\geo_{n;\omega}(k_j(\log n)^\beta)]+lD_{k_j},(\log n)^{2\beta})). 
\end{equation*}
\end{definition}

Figure~\ref{tunneling} shows how tunneling tubes look like.
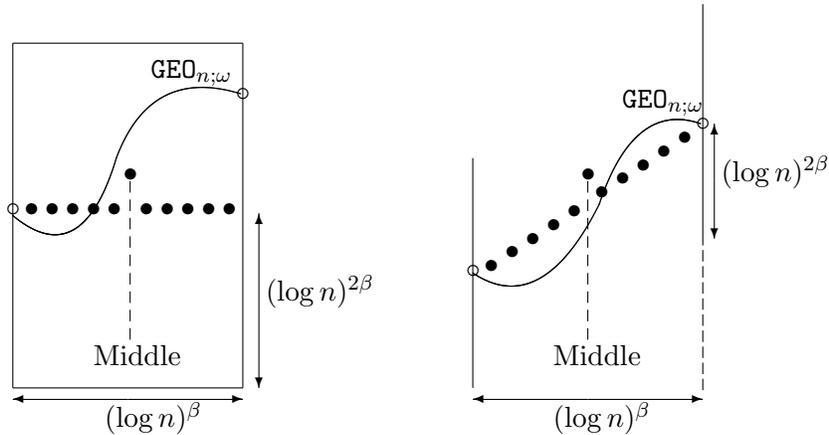
\begin{figure}[b]
\setlength{\unitlength}{0.87pt}
 \begin{picture}(200,200)(-100,-20)
 \put(-170,0){\line(0,1){150}}
 \put(-70,0){\line(0,1){150}}
 \put(-170,0){\line(1,0){100}}
 \put(-170,150){\line(1,0){100}}
 \multiput(-165,75)(9,0){5}{$\bullet$}
 \multiput(-115,75)(9,0){5}{$\bullet$}
 \multiput(-119,90)(0,-8){9}{\line(0,-1){5}}
 \put(-135,10){Middle}
 \put(-122,90){$\bullet$}
 \put(-173,75){$\circ$}
 \put(-73,125){$\circ$}
 \put(-120,-5){\vector(1,0){50}}
 \put(-120,-5){\vector(-1,0){50}}
 \put(-130,-17){$(\log n)^\beta$}
 \put(-63,38){\vector(0,1){38}}
 \put(-63,38){\vector(0,-1){38}}
 \put(-60,38){$(\log n)^{2\beta}$}
 \qbezier(-170,75)(-140,50)(-125,100)
 \qbezier(-125,100)(-110,140)(-71,128)
 \put(-110,135){$\geo_{n;\omega}$}
 
 \put(30,0){\line(0,1){100}}
 \put(130,67){\line(0,1){100}}
 \put(30,0){\line(3,2){100}}
 \put(30,100){\line(3,2){100}}
 \multiput(35,50)(9,6){5}{$\bullet$}
 \multiput(83,82)(9,6){5}{$\bullet$}
 \multiput(130,67)(0,-7){10}{\line(0,-1){5}}
 \multiput(80,90)(0,-8){9}{\line(0,-1){5}}
 \put(77,90){$\bullet$}
 \put(27,48){$\circ$}
 \put(127,112){$\circ$}
 \put(80,-5){\vector(1,0){50}}
 \put(80,-5){\vector(-1,0){50}}
 \put(65,10){Middle}
 \put(65,-17){$(\log n)^\beta$}
 \put(135,90){\vector(0,1){25}}
 \put(135,90){\vector(0,-1){25}}
 \put(138,90){$(\log n)^{2\beta}$}

 \qbezier(30,50)(60,30)(85,80)
 \qbezier(85,80)(100,125)(129,115)
 \put(95,122){$\geo_{n;\omega}$}
 \end{picture}
\caption{Tunneling tubes for $\alpha<1$ (left) and $\alpha> 1$ (right). 
Except for the middle, there are 
points of $\tilde\omega$ (indicated by $\bullet$) near the straight line at 
the center
of the tube. The gap in the middle is close to $\frac{3}{2}(\log n)^\epsilon$. }
\label{tunneling} 
\end{figure}
A simple computation shows the following lemma. 
\begin{lemma}
\label{lem-tunnel}
There exist positive constants $c_2$ such that for sufficiently large $n$, 
for any tube of the form~\eqref{tube<1} or~\eqref{tube>1},
\begin{equation*}
\bar\P(\textrm{The tube is tunneling for }\tilde\omega)
\ge \exp\set{-c_2(\log n)^{(2d+1)\beta}}.
\end{equation*}
\end{lemma}
\begin{definition}
Let
\begin{equation}
\begin{split}
   T^{\wedge}_n(\omega)
 =\inf\biggl\{T_n(\gamma)\colon &\gamma(0)=0, 
 \gamma\textrm{ is open and }\label{passtime2}\\
 &\max_{1\le k\le n}|\gamma(k-1)-\gamma(k)|\le 2(\log n)^{\beta+\epsilon}
 \biggr\} 
\end{split}
\end{equation}
with the convention that $\inf\emptyset=\infty$. 
Corresponding geodesics
are denoted by 
$\geo^{\wedge}_{n;\omega}$. 
\end{definition}

Now we introduce the following event: 
\begin{equation*}
\begin{split}
 \bar{\mathcal{E}}=\Bigl\{(\omega,\sigma_\omega,\tilde\omega)\colon
 &\omega\in\mathcal{E},
 \geo^{\wedge}_{n;\omega}=\geo^{\wedge}_{n;[\omega,\tilde\omega]_{\sigma_\omega}}
 \textrm{ outside }\sigma_\omega \textrm{ and }\\
 &\textrm{all tubes in $\sigma_\omega$ are 
 tunneling for }\tilde\omega\Bigr\}.
\end{split}
\end{equation*}
\begin{lemma}
\label{4}
For $\epsilon$ sufficiently small depending on $\alpha$ and $\beta$,  
 \begin{equation*}
 \bar\P(\bar{\mathcal{E}})
 \ge \P(\mathcal{E})\exp\set{-c_2n^\theta (\log n)^{\beta(2d+1)}}.
 \end{equation*}
\end{lemma}
\begin{proof}
We show that the second condition in $\bar{\mathcal{E}}$ is a consequence of
the others. Once this is shown, we have
\begin{equation*}
 \bar\P(\bar{\mathcal{E}})
 =\bar{\E}\left[ 
 \bar{\P}(\textrm{All tubes in $\sigma_\omega$ are 
 tunneling for }\tilde\omega\mid\sigma_\omega)1_{\{\omega\in\mathcal{E}\}}
 \right].
\end{equation*}
Since $\sigma_\omega$ consists of $n^\theta$ tubes for any 
$\omega\in\mathcal{E}$, the above conditional probability is bounded from 
below by $\exp\{-c_2n^\theta (\log n)^{\beta(2d+1)}\}$
by Lemma~\ref{lem-tunnel} and the desired bound follows. 

Note first that $\geo^{\wedge}_{n;\omega}=\geo_{n;\omega}$ for 
$\omega\in\mathcal{E}$ due to~\eqref{small-jump}. 
Suppose $\tilde\omega$ provides a tunneling configuration for all tubes in 
$\sigma_\omega$. Then for each of tube in $\sigma_\omega$, there is only 
one path satisfying the constraint in~\eqref{passtime2} inside, 
which we call \emph{tunnel}. We show that 
$\geo^{\wedge}_{n;[\omega,\tilde\omega]_{\sigma_\omega}}$ is the path $\gamma$
obtained by changing $\geo^{\wedge}_{n;\omega}$ to the tunnels in 
$\sigma_\omega$. 

Let us first check that the above operation improves the passage time, that is,
$T_\gamma([\omega,\tilde\omega]_{\sigma_\omega})<T_n(\omega)$.  
When $\alpha<1$, a tunnel has jumps smaller than $2c_1^{2/\alpha}$ except 
for the middle two jumps and the last one. The middle jumps are bounded 
by $2(\log n)^\epsilon$. 
Also as $\omega\in\mathcal{E}$, endpoints of a tunnel differ
at most $(\log n)^{(\beta+\epsilon)}+2\sqrt{d}$ and hence the last jump is 
bounded by $2(\log n)^{(\beta+\epsilon)}$. 
Therefore the above path satisfies the restriction in~\eqref{passtime2} and, if $\epsilon$ is so small that $\alpha(\beta+\epsilon)<\beta$,
its passage time is less than 
\begin{equation*}
 2^\alpha\left[c_1^2(\log n)^\beta
 +2(\log n)^{\alpha\epsilon}+(\log n)^{\alpha(\beta+\epsilon)}\right]
< c_1(\log n)^\beta/2 
\end{equation*}
for large $n$, by letting $c_1$ small if necessary. 
Since the tube in $\sigma_\omega$ is assumed to be black, this is strictly
smaller than the original passage time. 
The argument for $\alpha>1$ is similar. Indeed, the jumps in the $j$-th tunnel are, except for the middle two, smaller than $|D_{k_j}|+2c_1^2$ ($D_{k_j}$ is defined in Defninition~\ref{tunnel}). 
Therefore, the above path again satisfies the restriction in~\eqref{passtime2} 
and its passage time is, if $\epsilon<(\beta/\alpha)\wedge (1/2)$, 
less than
\begin{equation*}
(|D_{k_j}|+2c_1^2)^\alpha(\log n)^\beta
 +2^{1+\alpha}(\log n)^{\alpha\epsilon}
< (|D_{k_j}|^\alpha+c_1) (\log n)^\beta.
\end{equation*}
This is smaller than the original passage time by the definition of the
black box. 

Next we show that $\gamma$ is the best among the paths satisfying the constraint
in~\eqref{passtime2}. Any path which does not use the points in $\tilde\omega$ 
has a passage time larger than $T_n(\omega)$ and by the previous
step, it is not a geodesic for $[\omega,\tilde\omega]_{\sigma_\omega}$. 
On the other hand, the only way for a path satisfying the constraint 
in~\eqref{passtime2} to enter a tube $\tau\in\sigma_\omega$ is to share 
the left and right endpoints with $\geo_{n;\omega}$, which we call $l_\tau$ 
and $r_\tau$ respectively, since any other way
forces the path to make a jump larger than 
$(\log n)^{2\beta}/2\gg 2(\log n)^{\beta+\epsilon}$
(here we use the fact that the ``slope'' $D_{k_j}$ of tunnels are bounded by $M$). Repeating this argument for the passage time from $0$ to $l_\tau$ and 
from $r_\tau$ to $\{n\}\times \R^d$, we see that 
$\geo^{\wedge}_{n;[\omega,\tilde\omega]_{\sigma_\omega}}$ has to go through 
all the tubes in $\sigma_\omega$ sharing the left and right endpoints with 
$\geo_{n;\omega}$. This implies 
$\gamma=\geo^{\wedge}_{n;[\omega,\tilde\omega]_{\sigma_\omega}}$ and hence
$\geo^{\wedge}_{n;\omega}=\geo^{\wedge}_{n;[\omega,\tilde\omega]_{\sigma_\omega}}$ outside $\sigma_\omega$. 
\end{proof}
\begin{lemma}
\label{5}
There exists $c_3>0$ such that 
\begin{equation*}
 \bar\P(\bar{\mathcal{E}})
 \le \exp\{-c_3 n^\theta\log n\}.
 \end{equation*}
\end{lemma}
\begin{proof}
For any $S\in \{\sigma_\omega\colon \omega\in\mathcal{E}\}$, we define
\begin{equation*}
\begin{split}
\mathcal{E}_S=\bigl\{(\omega_1,\omega_2)\colon&
 \omega_1\in\mathcal{E}, 
 S\in\Sigma_{\omega_1}, \omega_1=\omega_2\textrm{ and }
 \geo^{\wedge}_{n;\omega_1}=\geo^{\wedge}_{n;\omega_2}
 \textrm{ outside }S,\\
&\textrm{ and all tubes in }S\textrm{ are tunneling for }\omega_2\bigr\}.
\end{split}
\end{equation*}
Then we have 
\begin{equation*}
\begin{split}
\bar\P(\bar{\mathcal{E}})
&=\int \P^{\otimes2}(\dd\omega,\dd\tilde\omega)\Q_\omega(\dd\sigma)
 \sum_S 1_{\{\sigma=S\}}
 1_{\{(\omega,[\omega,\tilde\omega]_S)\in\mathcal{E}_S\}}\\
&\le \frac{1}{\min_{\omega\in\mathcal{E}}\#\Sigma_\omega}
 \int \P^{\otimes2}(\dd\omega,\dd\tilde\omega) 
 \sum_{S} 1_{\{([\omega,\tilde\omega]_S,\omega)\in\mathcal{E}_S\}},
\end{split}
\end{equation*}
where in the second inequality, we have used the fact that
$(\omega,[\omega,\tilde\omega]_S)$ has the same law as 
$([\omega,\tilde\omega]_S,\omega)$. 
We shall show that the sum on the right hand side does not exceed one,
that is, 
$\{(\omega,\tilde\omega)\colon([\omega,\tilde\omega]_S,\omega)\in\mathcal{E}_S\}$ 
are disjoint for different $S$. 

Suppose that $([\omega,\tilde\omega]_S,\omega)\in\mathcal{E}_S$ and 
$([\omega,\tilde\omega]_{S'},\omega)\in \mathcal{E}_{S'}$ for $S\neq S'$. 
Then $S\in\Sigma_{[\omega,\tilde\omega]_S}$ and thus $S$ consists of 
$n^\theta$ tubes crossed by $\geo^{\wedge}_{n;[\omega,\tilde\omega]_S}$ 
as shown in Figure~\ref{resampling}. Noting that 
$\geo^{\wedge}_{n;[\omega,\tilde\omega]_S}=\geo^{\wedge}_{n;\omega}$ outside $S$, 
it follows that $\geo^{\wedge}_{n;\omega}$ enters and exits the tubes in $S$
as if $S\in\Sigma_\omega$. 
Since these observations apply for $S'$  as well, it follows that $S'$ contains 
a tube disjoint with $S$ located along $\geo^{\wedge}_{n;\omega}$ (since both 
$S$ and $S'$ consist of the same number of tubes located along 
$\geo^{\wedge}_{n;\omega}$). 
However this contradicts $[\omega,\tilde\omega]_S\in\mathcal{E}$ as follows: 
note that $\geo^{\wedge}_{n;[\omega,\tilde\omega]_S}=\geo^{\wedge}_{n;\omega}$ 
in the tube found above ($\not\in S$) which is assumed to be tunneling 
for $\omega$. Then $\geo^{\wedge}_{n;[\omega,\tilde\omega]_S}$ must follow the 
tunnel (see the proof of Lemma~\ref{4} for the terminology), 
and hence it makes a jump larger than $(\log n)^\epsilon$ there.

Recalling~\eqref{cardinality}, we arrive at 
\begin{equation*}
\bar\P(\bar{\mathcal{E}})
\le \frac{1}{\min_{\omega\in\mathcal{E}}\#\Sigma_\omega}
\le \exp\set{-(1-2\theta)n^\theta\log n}.
\end{equation*}
\end{proof}
\begin{proof}
 [Proof of Theorem~\ref{thm:main}]
Combining Lemmas~\ref{4} and~\ref{5}, we obtain for $\beta<(2d+1)^{-1}$ that
\begin{equation*}
\P(\mathcal{E})
 \le \exp\{- n^\theta\}. 
\end{equation*}
Further invoking Lemma~\ref{3}, we arrive at
\begin{equation*}
\begin{split}
\P(\eqref{small-jump})
&\le \P(\mathcal{E})+\P(\textrm{\eqref{small-jump} 
 holds and either \eqref{many-black} or \eqref{linear} fails})\\
&\qquad+\P(\exists x,y\in \omega\cap
[0,n]\times[-n^{1+\epsilon},n^{1+\epsilon}]
\textrm{ with } |x-y|< e^{-n})\\
&\le \exp\{-n^\theta\}+\exp\{-n^{1\wedge\frac{d}{\alpha}}(\log n)^{-\beta}\}
+e^{-cn}.
\end{split}
\end{equation*}
This and the Borel--Cantelli lemma complete the proof.
\end{proof}

\section{Numerical experiments}
\label{sec:numerics}
In this section, some results from a numerical experiment are presented. We first see how the geodesics look like for $\alpha<1$ and $\alpha>1$. Keeping track of geodesics requires a large memory and it is done only for 1024 steps. Then we proceed to statistical studies. For this purpose, about 1300 ($n=4096$), 1000 ($n=8192$), 700 ($n=16382$), and 400 ($n=32764$) samples for nine values of $\alpha$ ranging from 0.5 to 1.3 are generated numerically. Using these data, we see how the maximal jump size, the maximal displacement, and the variance of the passage time behave as functions of $n$. 

Figure~\ref{fig:opt} shows the geodesics in 1024 steps with $\alpha=0.6$ and $\alpha=1.2$. It makes a large jump at the beginning when $\alpha=0.6$ while it is much smoother when $\alpha=1.2$. The reason can be explained as follows. The motivation for the geodesic to make a large displacement is to go to a ``good point'' from which the passage time is atypically small. Since the fluctuation of the passage time grows in time, there is a better chance to find a good point earlier than later. Now if there is a good point $(k,x)$ for relatively small $k$, then $x$ is typically far from the origin. The shortest path to go from $(0,0)$ to such a point $(k,x)$ tries to make one large jump and otherwise mostly horizontal when $\alpha<1$, and tries to stay close to the straight line between $(0,0)$ and $(k,x)$ when $\alpha>1$. 

\begin{figure}[h]
\includegraphics[width=320pt]{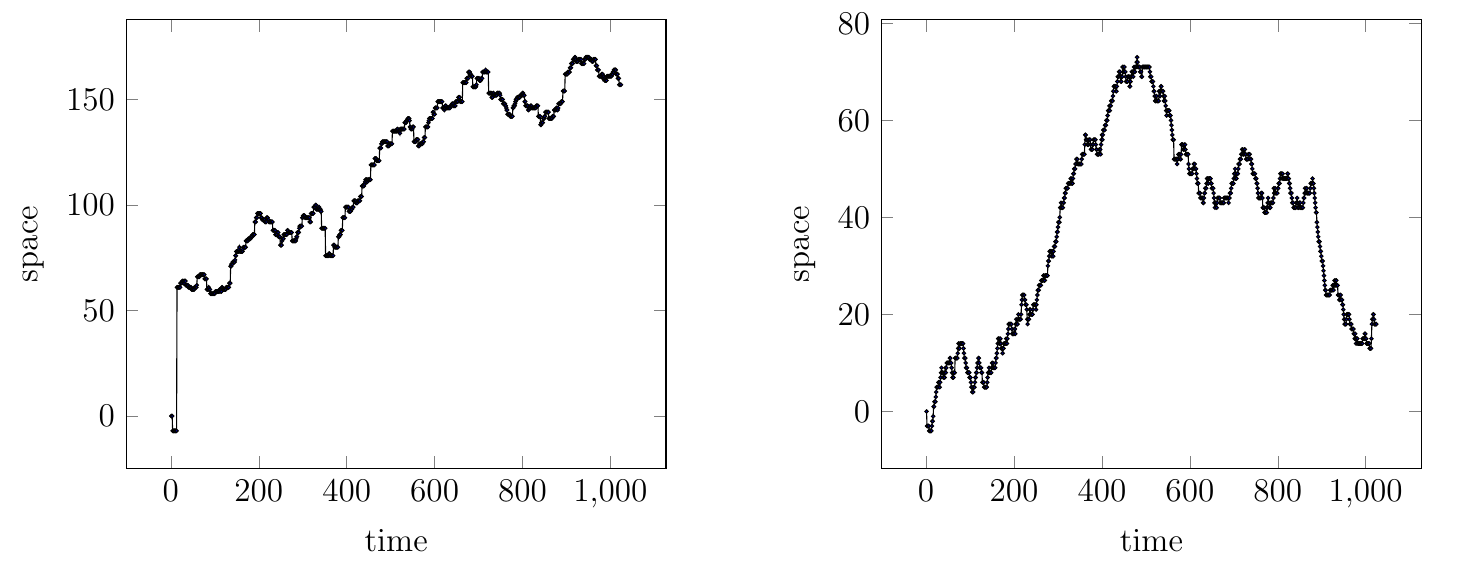}
\caption{The geodesics in 1024 steps with $\alpha=0.6$ (left), whose maximal jump size is 68, and with $\alpha=1.2$ (right) whose maximal jump size is 4.}
\label{fig:opt}
\end{figure}

Table~\ref{table:mean} shows the maximal jump sizes of geodesic. The size decreases rapidly as $\alpha$ approaches to 1. Assuming that the maximal jump size obeys a power law, their exponents are computed. We know from~\cite[Corollary~2.7]{Nak18} that this is not the case for $\alpha>1$. But the interesting regime here is $\alpha<1$, where the exponents are larger. This suggest that $(\log n)^\epsilon$ in Theorem~\ref{thm:main} can be improved to $n^\epsilon$. We leave this as a conjecture. 

\begin{table}[h]
\centering 
\begin{tabular}{c c c c c c}
\hline                      
$\alpha$ & $n=4096$ & $n=8192$ & $n=16382$ & $n=32764$ & Exponent\\
\hline                  
0.5 & 128 & 218 & 314 & 434 & 0.586\\
0.6 & 46.7 & 57.6 & 74.5 & 89.1 & 0.311\\
0.7 & 19.6 & 23.1  & 28.1 & 31.7 & 0.232\\
0.8 & 12.2 & 13.3 & 14.9 & 16.8 & 0.155\\
0.9 & 9.02 & 9.88 & 10.7 & 11.5 & 0.116\\       
1.0 & 5.59 & 5.97 & 6.44 & 6.76 & 0.091\\       
1.1 & 3.57 & 3.82 & 4.14 & 4.45 & 0.106\\       
1.2 & 3.49 & 3.69 & 4.03 & 4.24 & 0.094\\       
1.3 & 3.27 & 3.45 & 3.73 & 3.95 & 0.091\\       
\hline\\
\end{tabular}
\caption{Empirical means of the maximal jump size. The exponents in the rightmost column are read from the slopes in log-log plot.}
\label{table:mean}
\end{table}

Tables~\ref{table:displacement} and~\ref{table:fluctuation} show the maximal displacement and the variance, respectively. The maximal displacement decreases in $\alpha$ while the variance increases. This is natural since smaller $\alpha$ allows the path to explore better environments in a wider area, and it should cause a stronger self-averaging. 
The so-called fluctuation exponent $\chi$ and the wondering exponent $\xi$ are computed, again assuming the power law behaviors. The results are not far from the Kardar--Parisi--Zhang prediction $2\chi=\xi=2/3$, at least when $\alpha>1$. 

It would be desirable to do a numerical experiment in a larger scale to make sharper predictions. \\

\begin{table}[h]
\centering 
\begin{tabular}{c c c c c c}
\hline                      
$\alpha$ & $n=4096$ & $n=8192$ & $n=16382$ & $n=32764$ & $\xi$\\
\hline                  
0.5 & 432 & 739 & 1200 & 1870 & 0.704\\
0.6 & 277 & 448 & 713 & 1110 & 0.666\\
0.7 & 216 & 350  & 553 & 869 & 0.669\\
0.8 & 193 & 303 & 479 & 767 & 0.664\\
0.9 & 179 & 280 & 447 & 703 & 0.658\\       
1.0 & 171 & 264 & 411 & 662 & 0.651\\       
1.1 & 169 & 260 & 408 & 634 & 0.646\\       
1.2 & 164 & 264 & 423 & 652 & 0.664\\       
1.3 & 166 & 265 & 410 & 662 & 0.665\\       
\hline\\
\end{tabular}
\caption{Empirical means of the maximal displacement. The wondering exponents $\xi$ are read from the slopes in log-log plot.}
\label{table:displacement}
\end{table}
\begin{table}[h]
\centering 
\begin{tabular}{c c c c c c}
\hline                      
$\alpha$ & $n=4096$ & $n=8192$ & $n=16382$ & $n=32764$ & $2\chi$\\
\hline                  
0.5 & 43.9 & 67.1 & 95.0 & 158 & 0.615\\
0.6 & 58 & 102 & 162 & 255 & 0.712\\
0.7 & 77 & 123  & 204 & 354 & 0.733\\
0.8 & 104 & 158 & 269 & 394 & 0.641\\
0.9 & 113 & 201 & 317 & 490 & 0.705\\       
1.0 & 154 & 252 & 374 & 641 & 0.685\\       
1.1 & 164 & 241 & 438 & 624 & 0.624\\       
1.2 & 183 & 267 & 459 & 721 & 0.659\\       
1.3 & 176 & 287 & 467 & 747 & 0.695\\       
\hline\\
\end{tabular}
\caption{The variance of the passage time. The fluctuation exponents $\chi$ are read from the slopes in log-log plot.}
\label{table:fluctuation}
\end{table}

\textbf{Acknowledgments.}
The author is indebted to Hugo Duminil-Copin for explaining the content of the paper~\cite{DCKNPS} prior to its publication. He also thanks Nobuo Yoshida for carefully reading an early version of this article. Part of this work has been done during the author's stay at ENS Paris in 2016 October. He thanks ENS for hospitality and Mitsubishi Heavy Industries for the financial support. This work was also supported by JSPS KAKENHI Grant Number JP24740055 and ISHIZUE 2019 of Kyoto University Research Development Program.

\newcommand{\noop}[1]{}\def\cprime{$'$}


\end{document}